\documentclass[10pt]{article}
\usepackage[matrix,arrow,curve]{xy}
\usepackage{amssymb, amsthm}
\usepackage{supertabular}
\usepackage{amsfonts}
\usepackage{latexsym}
\usepackage{enumerate}
\usepackage[dvips]{graphics}

\newtheorem{dummy}{anything}
\newtheorem{Theorem}[dummy]{Theorem}
\newtheorem{Lemma}{Lemma}

\newtheorem{Definition}{Definition}

\newcommand{\cals}{{\mathcal S}}

\newcommand{\R}{{\mathbf R}}

\newcommand{\Z}{{\mathbf Z}}

\newcommand{\coker}{{\mbox{\rm coker}}}
\newcommand{\rk}{{\rm {rk}}}
\newcommand{\vol}{\rm {vol}}

\newcommand{\pcirc}{\kern .7pt {\scriptstyle \circ} \kern 1pt}

\newcommand{\eqref}[1]{(\ref{#1})}
\newcommand{\hfl}[2]{\smash{\mathop{\hbox to 1 truecm{\kern %
3pt\rightarrowfill\kern 3pt}}%
\limits^{\scriptstyle#1}_{\scriptstyle#2}}}
\newcommand{\cqfd}{\unskip\kern 6pt\penalty 500%
\raise -2pt\hbox{\vrule\vbox to10pt{\hrule width %
4pt\vfill\hrule}\vrule}\smallskip}

\title{On the conjecture of Kevin Walker}
\author{Michael Farber\\\small Department of Mathematical Sciences, University of Durham, UK \\\small email: Michael.Farber@durham.ac.uk, Fax: 0044-191-3343051
\and Jean-Claude Hausmann\\ \small Section mathematique, Universit\'e de
Gen\`eve, Suisse\\ \small email: hausmann@math.unige.ch, Fax: 0041-22-3791176
\and Dirk Sch\"utz\\ \small Department of Mathematical Sciences, University of Durham, UK\\ \small email: Dirk.Schuetz@durham.ac.uk, Fax: 0044-191-3343051}
\date{August 22, 2007}

\begin{document}
\maketitle 

\begin{abstract}
In 1985 Kevin Walker  in his study of
 topology of polygon
spaces \cite{Wa} raised an interesting conjecture in the spirit of
the well-known question \lq\lq Can you hear the shape of a
drum?\rq\rq\ of Marc Kac. Roughly, Walker's conjecture asks if one
can recover relative lengths of the bars of a linkage from
intrinsic algebraic properties of the cohomology algebra of its
configuration space. In this paper we prove that the conjecture is true for polygon spaces in $\R^3$.
We also prove that for planar polygon spaces the conjecture holds is several modified forms:
(a) if one takes into account the action of a natural involution on cohomology, (b) if the cohomology algebra of the involution's orbit space is known, or (c) if the length vector is normal. Some of our results allow the length vector to be non-generic, the corresponding polygon spaces have singularities. Our main tool is the study of the natural involution and its action on cohomology. A crucial role in our proof plays the solution of the isomorphism problem for monoidal rings due to J. Gubeladze.
\end{abstract}

\section{Introduction}

Let $\ell=(l_1, \dots, l_n)$ be a sequence positive real numbers.
The planar polygon space $M_\ell$ parameterizes shapes of all planar
$n$-gons having sides of length $l_1, \dots, l_n$. Formally,
$M_\ell$ is defined as
\begin{eqnarray}\label{ml}
M_\ell \, = \, \{(u_1, \dots, u_n)\in S^1\times\dots\times S^1; \,
\, \sum_{i=1}^n l_iu_i=0\}/{{\rm SO}(2)}
\end{eqnarray}
where the group ${\rm SO}(2)$ acts diagonally on the product
$S^1\times\dots\times S^1$. If the length vector $\ell$ is generic
(see below) then $M_\ell$ is a closed smooth manifold of dimension
$n-3$. If $\ell$ is not generic then $M_\ell$ is a smooth compact
$(n-3)$-manifold having finitely many singular points.

The polygon space $M_\ell$ emerges in topological robotics as the
configuration space of the planar linkage, a simple mechanism
consisting of $n$ bars of length $l_1, \dots, l_n$ connected by
revolving joints forming a closed planar polygonal chain. The
significance of spaces $M_\ell$ was also recognized in molecular
biology where these spaces describe shapes of closed molecular
chains. Statistical shape theory, see e.g. \cite{Kendall, HR}, is another area where the spaces
$M_\ell$ play an interesting role: they describe the space of shapes having
certain geometric properties with respect to the central point.

Inspired by the work of W.~Thurston and J.~Weeks on linkages\footnote{See also thesis of S. H. Niemann \cite{N} written in Oxford in 1978.}
\cite{TW}, the configuration space of generic polygon spaces were studied
by K.~Walker \cite{Wa}, A.~A.~Klyachko \cite{Kl}, M.~Kapovich and
J.~Millson \cite{KM1}, J.-Cl. Hausmann and A. Knutson \cite{HK}
and others. Non-generic polygon spaces were studied by A.~Wenger
\cite{We} and the Japanese school (see, e.g.~\cite{Ka}).
The Betti numbers of $M_\ell$ as functions of the length
vector $\ell$ are described in \cite{FS}; in particular, it was
shown in \cite{FS} that the integral homology $H_\ast(M_\ell)$ has
no $\Z$-torsion. A.~A.~Klyachko \cite{Kl} found the Betti numbers
of spatial polygon spaces and their integral cohomology ring is given
in~\cite{HK}.

The spatial polygon spaces $N_\ell$ (defined by equation (\ref{nell}) below) emerge also as spaces of semistable configurations of $n$ points in ${\mathbf {CP}}^1$ having weights $l_1, \dots, l_n$, modulo M\"obius transformations, see \cite{Kl}, \cite{KM}. Planar polygon spaces $M_\ell$ admit a similar description
with the real projective line ${\mathbf {RP}}^1$ replacing the complex projective line.

Walker's conjecture \cite{Wa} states that for a generic $\ell$ the
cohomology ring of $M_\ell$ determines the length vector $\ell$ up
to a natural equivalence (described below). In this paper we prove several results confirming this conjecture.
Firstly, we show
that this statement is indeed true if one takes into account a
natural involution acting on the cohomology. Our main result in this direction goes
a little further by allowing also non-generic length vectors. Secondly we show that Walker's conjecture in its original form holds
for normal length vectors. We also prove results in the spirit of the Walker conjecture for the spatial polygon spaces.

To state the Walker conjecture in full detail we need to explain the dependance of the configuration space $M_\ell$ on the
length vector
\begin{eqnarray}
\ell=(l_1, \dots,l_n)\in \R^n_+. \end{eqnarray} Here $\R^n_+$
denotes the set of vectors in $\R^n$ having nonnegative
coordinates. Clearly, $M_\ell=M_{t\ell}$ for any $t>0$. Also, it is easy to see that $M_\ell$ is diffeomorphic to $M_{\ell'}$ if
$\ell'$ is obtained from $\ell$ by permuting coordinates.

Denote by $A=A^{(n-1)}\subset \R^n_+$ the interior of the unit
simplex, i.e. the set given by the inequalities $l_1 >0, \dots,
l_n>0$, $\sum l_i =1$. One can view $A$ as the quotient space of $\R^n_+$ with respect to $\R_+$-action.
For any subset $J\subset \{1, \dots, n\}$
we denote by $H_J\subset \R^n$ the hyperplane defined by the
equation
\begin{eqnarray}\label{wall}\sum_{i\in J} l_i = \sum_{i\notin J}
l_i.\end{eqnarray}
One considers the following stratification
\begin{eqnarray}\label{strat}
A^{(0)}\subset A^{(1)}\subset \dots \subset A^{(n-1)}=A
\end{eqnarray}
Here the symbol $A^{(i)}$ denotes the set of points $\ell\in A$
lying in $\geq n-1-i$ linearly independent hyperplanes $H_J$ for
various subsets $J$. A {\it stratum} of dimension $k$ is a
connected component of the complement $A^{(k)}-A^{(k-1)}$. By
Theorem 1.1 of \cite{HR}, the smooth spaces\footnote{i.e. manifolds with singularities.} $M_\ell$ and $M_{\ell'}$
are diffeomorphic if the vectors $\ell$ and $\ell'$ belong to the
same stratum.

Strata of dimension $n-1$ are called {\it chambers}. A vector
$\ell$ lying in a chamber is called {\it generic}. Non-generic
length vectors lie in walls separating chambers and hence satisfy
linear equations (\ref{wall}) for some $J$. A linkage with
non-generic length vector is characterized by the property of
allowing collinear configurations, i.e. such that all its bars are
parallel to each other.

{\bf Walker's conjecture:} {\it Let $\ell, \ell'\in A$ be two generic
length vectors; if the corresponding polygon spaces $M_\ell$ and
$M_{\ell'}$ have isomorphic graded integral cohomology rings then
for some permutation $\sigma: \{1, \dots, n\} \to \{1, \dots, n\}$
the length vectors $\ell$ and $\sigma(\ell')$ lie in the same
chamber of $A$.}

In this paper we prove that Walker's conjecture holds for polygon spaces in $\R^3$ (see Theorem \ref{thmain31})
assuming that the number of links $n$ is distinct from $4$.
We also prove that the conjecture is true for planar polygon spaces in several slightly modified forms, see Theorems \ref{thmain}, \ref{thmain21}, \ref{thmain2}.

\section{Statements of the main results}

Polygon spaces (\ref{ml}) come with a natural
involution \begin{eqnarray}\label{invol} \tau: M_\ell \to M_\ell,
\quad \tau(u_1, \dots, u_n) = (\bar u_1, \dots, \bar u_n)
\end{eqnarray} induced by complex conjugation.
Geometrically, this involution associates to a polygonal shape the
shape of the reflected polygon. The fixed points of $\tau$ are the
collinear configurations, i.e. degenerate polygons. In particular
we see that $\tau: M_\ell\to M_\ell$ has no fixed points iff the
length vector $\ell$ is generic. Clearly, $\tau$ induces an
involution \begin{eqnarray}\label{involstar}
\tau^\ast:
H^\ast(M_\ell) \to H^\ast(M_\ell)
\end{eqnarray} on the cohomology
of $M_\ell$ with integral coefficients.\footnote{In this paper we adopt the following convention: we skip $\Z$ from notation when dealing with homology and cohomology with integral coefficients. We indicate explicitly all other coefficient groups.}

\begin{Theorem}\label{thmain} Suppose that two length vectors $\ell, \ell'\in
A^{(n-1)}$ are ordered, i.e. $\ell=(l_1, l_2, \dots, l_n)$ with $0<l_1\leq l_2\leq \dots \leq l_n$ and similarly for $\ell'$. If
there exists a graded ring isomorphism of the integral
cohomology algebras $$f: H^\ast(M_\ell) \to H^\ast(M_{\ell'})$$
commuting with the action of the involution (\ref{involstar})
then $\ell$ and
$\ell'$ lie in the same stratum of $A$. In particular,
under the above assumptions the moduli spaces $M_\ell$ and
$M_{\ell'}$ are $\tau$-equivariantly diffeomorphic.
\end{Theorem}

Let $\ell,\ell'\in A^{(n-1)}$ be length vectors (possibly non-generic) lying in the same stratum.
Then there exists a diffeomorphism\footnote{"Diffeomorphism" refers to the smooth structure on $M_{\ell}$,
see \cite[Section~2.2]{HR}} $\phi:M_{\ell'}{\to}
M_{\ell}$
which is equivariant with respect to the involution $\tau$.
Such a diffeomorphism was constructed in the proof of \cite[Theorem~1.1]{HR}, see \cite[p.~36]{HR}; compare \cite[Remark~3.3]{HR}.

Let $\bar M_\ell$ denote the factor-space of $M_\ell$ with respect to the involution (\ref{invol}). An alternative definition of $\bar M_\ell$ is given by
\begin{eqnarray}\label{mlbar}
\bar M_\ell \, = \, \{(u_1, \dots, u_n)\in S^1\times\dots\times S^1; \,
\, \sum_{i=1}^n l_iu_i=0\}/{{\rm O}(2)}.
\end{eqnarray}

\begin{Theorem}\label{thmain21} Suppose that two generic ordered length vectors $\ell, \ell'\in
A^{n-1}$ are such that there exists a graded algebra isomorphism
$$f: H^\ast(\bar M_\ell;\Z_2) \to H^\ast(\bar M_{\ell'};\Z_2)$$ of cohomology algebras with $\Z_2$ coefficients.
If $n\not= 4$ then $\ell$ and
$\ell'$ lie in the same chamber of $A$.
\end{Theorem}

Theorem \ref{thmain21} is false for $n=4$. Indeed, for the length vectors $\ell=(1,1,1,2)$
and $\ell'=(1,2,2,2)$ the manifolds  $\bar M_\ell$ and $\bar M_{\ell'}$ are circles. However $M_\ell$ and $M_{\ell'}$ are not diffeomorphic (the first is $S^1$ and the second is $S^1\sqcup S^1$) and thus $\ell$ and $\ell'$ do not lie in the same chamber of $A$.

In this paper we also prove a result in the spirit of Walker's conjecture for the spatial polygon spaces. These spaces are defined by
\begin{eqnarray}\label{nell}
N_\ell \, = \, \{(u_1, \dots, u_n)\in S^2\times\dots\times S^2; \,
\, \sum_{i=1}^n l_iu_i=0\}/{{\rm SO}(3)}.
\end{eqnarray}
Points of $N_\ell$ parameterize the shapes of $n$-gons in $\R^3$ having sides of length $\ell=(l_1, \dots, l_n)$.
If the length vector $\ell$ is generic then $N_\ell$ is a closed smooth manifold of dimension $2(n-3)$.

\begin{Theorem}\label{thmain31}
 Suppose that two generic ordered length vectors $\ell, \ell'\in
A^{(n-1)}$ are such that there exists a graded algebra isomorphism
$$f: H^\ast(N_\ell;\Z_2) \to H^\ast(N_{\ell'};\Z_2)$$ of cohomology algebras with $\Z_2$ coefficients.
If $n\not= 4$ then $\ell$ and
$\ell'$ lie in the same chamber of $A$. This theorem remains true if the cohomology algebras are taken with integral coefficients.
\end{Theorem}

Theorem \ref{thmain31} is false for $n=4$: for length vectors $\ell=(1,1,1,2)$
and $\ell'=(1,2,2,2)$ lying in different chambers (see above) the manifolds  $N_\ell$ and $N_{\ell'}$ are both diffeomorphic to $S^2$.

To state another main result of this paper we need to
introduce some terminology. First we recall the notions of short, long and median subsets, see
\cite{HK}, \cite{FS}. Given a length vector $\ell=(l_1, l_2,
\dots, l_n)$, a subset of the set of indices $J\subset \{1, 2,
\dots, n\}$ is called {\it short} if $$\sum_{i\in J} l_i <
\sum_{i\notin J} l_i.$$ The complement of a short subset is called
{\it long}. A subset $J$ is called {\it median} if $$\sum_{i\in J}
l_i = \sum_{i\notin J} l_i.$$

The following simple fact relates these notions to stratification (\ref{strat}).

\begin{Lemma} \label{stra}
Two length vectors $\ell, \ell'\in A^{(n-1)}$ lie in the same stratum of $A^{(n-1)}$ if and only if, for all subsets $J\subset \{1, 2, \dots, n\}$,
$J$ is short with respect to $\ell$ if and only if it is short with respect to $\ell'$.
\end{Lemma}
\begin{proof} $J$ is long iff the complement $\bar J$ is short and $J$ is median iff neither $J$ nor $\bar J$ is short.
Hence, vectors $\ell, \ell'$ satisfying conditions of the lemma have identical families of short, long and median subsets. This clearly implies that
$\ell$ and $\ell'$ lie in the same stratum of $A$.
\end{proof}

\begin{Definition}\label{defnormal}
A length vector $\ell=(l_1, \dots, l_n)$ is called normal if \begin{eqnarray}\label{inter}\cap J\, \not=\, \emptyset \end{eqnarray}
where $J$ runs over all subsets $J\subset \{1, \dots, n\}$ with $|J|=3$ which are long with respect to $\ell$.
A stratum of $A^{(n-1)}$ is called normal if it contains a normal vector.
\end{Definition}

Clearly, any vector lying in a normal stratum is normal.
A length vector $\ell$ with the property that all subsets $J$ of cardinality $3$ are short or median
is normal
 since then the intersection (\ref{inter})
equals $\{1, \dots, n\}$ as the intersection of the empty family.

If $\ell=(l_1, \dots, l_n)$ where $0<l_1\le l_2\le \dots \le l_n$
then $\ell$ is normal if and only if the set $\{n-3, n-2, n-1\}$
is short or median with respect to $\ell$. Indeed, if this set is long then the sets $\{n-3, n-2,
n\}$, $\{n-3, n-1, n\}$, $\{n-2, n-1, n\}$ are long as well and
the intersection of these four sets of cardinality three is empty.
On the other hand, if the set $\{n-3, n-2, n-1\}$ is short or median then
any long subset of cardinality three $J\subset \{1, \dots, n\}$ contains $n$ and therefore
the intersection (\ref{inter}) also
contains $n$.


Examples of non-normal length vectors are $(1,1,1,1,1)$ (for
$n=5$) and $(1,1,2,2,2,3)$ for $n=6$. Only 7 chambers out 21 are normal for $n=6$.
However, for large $n$ it is very likely that a randomly
selected length vector is normal. For $n=9$, where there are
175429 chambers up to permutation, 86\% of them
are normal. It is shown in \cite{Fa} that the $(n-1)$-dimensional volume of the union $\mathcal N_n\subset A^{n-1}$ of all normal chambers satisfies
$$\frac{\vol(A^{n-1} - {\mathcal N}_n)}{\vol(A^{n-1})} \, <\,  \frac{24 n^6}{2^n},$$
i.e. for large $n$ the relative volume of the union of non-normal chambers is exponentially small.

\begin{Theorem}\label{thmain2}
Suppose that $\ell, \ell'\in A^{(n-1)}$ are two ordered length
vectors such that there exists a graded algebra isomorphism
between the integral cohomology algebras $H^\ast(M_\ell) \to
H^\ast(M_{\ell'})$. Assume that one of the vectors $\ell,
\ell'$ is normal. Then the other vector is normal as well and
$\ell$ and $\ell'$ lie in the same
stratum of the simplex $A$.
\end{Theorem}

Consider the action of the symmetric group $\Sigma_n$ on the simplex $A^{(n-1)}$ induced by permutations of vertices.
This action defines an action of $\Sigma_n$ on the set of strata and chambers
and we
denote by $c_n$ and by $c^\ast_n$
the number of distinct
$\Sigma_n$-orbits of chambers (or chambers consisting of normal length vectors, respectively).

Theorems \ref{thmain} - \ref{thmain2} imply:
\begin{Theorem}
(a) For $n\not=4$ the number of distinct diffeomorphism types of manifolds $N_\ell$, where $\ell$ runs over all generic vectors of $A^{(n-1)}$, equals $c_n$;

(b) for $n\not=4$ the number of distinct diffeomorphism types of manifolds $\bar M_\ell$, where $\ell$ runs over all generic vectors of $A^{(n-1)}$, equals $c_n$;

(c) the number $x_n$ of distinct diffeomorphism types of manifolds $M_\ell$, where $\ell$ runs over all generic vectors of $A^{(n-1)}$, satisfies $c^\ast_n \leq x_n \leq c_n$.

(d) the number of distinct diffeomorphism types of manifolds with singularities $M_\ell$, where $\ell$ varies in $A^{(n-1)}$, is bounded above by the number of distinct $\Sigma_n$-orbits of strata of $A^{(n-1)}$ and is bounded below by the number of distinct $\Sigma_n$-orbits of normal strata of $A^{(n-1)}$.

Statements (a), (b), (c), (d) remain true if one replaces the words \lq\lq diffeomorphism types\rq\rq\, by \lq\lq homeomorphism types\rq\rq\, or by \lq\lq homotopy types\rq\rq.
\end{Theorem}

It is an interesting combinatorial problem to find explicit formulae
for the numbers $c_n$ and $c_n^*$ and to understand their behavior for large $n$ .
For $n\leq 9$, the numbers $c_n$ have
been determined in \cite{HR}, by giving an explicit list of the
chambers. The following table gives the values $c_n$ and $c_n^\ast$ for $n\le 9$:
\vskip 0.5cm
\begin{center}
\begin{tabular}{|c|c|c|c|c|c|c|c|}
\hline
$n$ &\footnotesize 3&\footnotesize 4&\footnotesize 5&\
\footnotesize 6 &\footnotesize 7&\footnotesize 8&\footnotesize 9
\\[1mm] \hline \hline
\small $c_n$ &\footnotesize 2 &\footnotesize 3&
\footnotesize 7&\footnotesize 21&\footnotesize 135&
\footnotesize 2470&  \footnotesize  175428
\\[1mm] \hline
\small $c_n^*$
& \footnotesize 1& \footnotesize 1& \footnotesize 2
& \footnotesize 7 & \footnotesize 65
& \footnotesize 1700 & \footnotesize 151317 \\
\hline
\end{tabular}
\end{center}
\vskip 0.2cm

Proofs of Theorems \ref{thmain} - \ref{thmain2} are given in \S \S \ref{prfthm1}, \ref{prfthm34}, \ref{prfthm2}. In \S \S 3, 4 we describe some
auxiliary results which are used in the proofs.

\section{The balanced subalgebra}\label{balanced}

In this section we will study the action of the involution (\ref{invol}) on the integral cohomology of planar polygon space $M_\ell$.

An integral cohomology class $u\in H^i(M_\ell)$ will be called
{\it balanced} if
\begin{eqnarray}
\tau^\ast(u) = (-1)^{\deg u} u.
\end{eqnarray}
The product of balanced cohomology classes is balanced. The set of
all balanced cohomology classes forms a graded subalgebra
\begin{eqnarray} B^\ast_\ell\subset H^\ast(M_\ell).\end{eqnarray}
In this section we describe explicitly the structure of the
balanced subalgebra in terms of the length vector $\ell$.

Let us assume that the length vector $\ell$ is {\it ordered}, i.e. it satisfies
\begin{eqnarray}\label{ordered}
0\leq l_1\leq l_2\leq \dots \leq l_n.
\end{eqnarray}
It is well known \cite{KM1} that $M_\ell$ is empty if and only if
the subset $J=\{n\}$ is long.

Next we describe specific cohomology classes
\begin{eqnarray}\label{generators}
X_1, X_2, \dots, X_{n-1}\, \in\,  H^1(M_\ell).
\end{eqnarray}
Consider the map $\phi_i: M_\ell \to S^1$ given by $$\phi_i(u_1, u_2, \dots, u_n)= u_iu_n^{-1}\in S^1.$$ Here $i=1, 2, \dots, n$ and $\phi_i$ associates to a configuration the angle between the links number $i$ and $n$. We denote by $X_i\in H^1(M_\ell)$ the induced class $X_i=\phi^\ast_i([S^1])$ where $[S^1]$ denote the fundamental class of the circle oriented in the usual anticlockwise manner.

Since the complex conjugation reverses the orientation of the circle, we obtain that $\tau^\ast(X_i)=-X_i$, i.e. cohomology classes $X_i$ are balanced.

\begin{Theorem}\label{thm2}
Assume that $\ell=(l_1, \dots, l_n)$
satisfies (\ref{ordered}) and the single element subset $\{n\}$ is
short. Then the balanced subalgebra $B_\ell^\ast$, viewed as a
graded skew-commutative ring, is generated by the classes $X_1, \dots, X_{n-1}\in H^1(M_\ell)$ and is isomorphic to the factor ring
$$E(X_1,\dots, X_{n-1})/I$$ where $E(X_1, \dots, X_{n-1})$ denotes
the exterior algebra having degree one generators $X_1, \dots,
X_{n-1}$ and $I\subset E(X_1, \dots, X_{n-1})$ denotes the ideal
generated by the monomials $$X_{r_1}X_{r_2}\dots X_{r_i},$$ one
for each sequence $1\leq r_1< r_2< r_3< \dots <r_i<n$ such that
the subset $$\{r_1, \dots, r_i\}\cup \{n\}\subset \{1, \dots,
n\}$$ is long.
\end{Theorem}

\begin{proof} Consider the robot arm with $n$ bars of length $l_1,
\dots, l_n$. It is a simple mechanism consisting of bars (links)
connected by revolving joints. The initial point of the robot arm
is fixed on the plane.
\begin{figure}[h]
\begin{center}
\resizebox{6cm}{4cm} {\includegraphics[90,368][440,632]{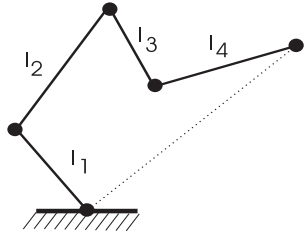}}
\end{center}
\caption{Robot arm.}\label{arm}
\end{figure}
The moduli space of a robot arm (i.e. the space of its possible
shapes) is
\begin{eqnarray}\label{w}
W=\{(u_1, \dots, u_n)\in S^1\times \dots\times S^1\}/{ S^1}.
\end{eqnarray}
Clearly, $W$ is diffeomorphic to a torus $T^{n-1}$ of dimension
$n-1$. A diffeomorphism can be specified, for example, by
assigning to a configuration $(u_1, \dots, u_n)$ the point
$(u_1u_n^{-1}, u_2u_n^{-1}, \dots, u_{n-1}u_n^{-1})\in T^{n-1}$
(measuring angles between the directions of the last and the other
links). The moduli space of closed polygons $M_\ell$ (where
$\ell=(l_1, \dots, l_n)$) is naturally embedded into $W$. We will
study the homomorphism \begin{eqnarray}\label{induced} j^\ast:
H^\ast(W) \to H^\ast(M_\ell)\end{eqnarray} induced by the
inclusion $j: M_\ell \to W$ on cohomology with integral
coefficients. Since $W=T^{n-1}$, the cohomology $H^\ast(W)$ is an
exterior algebra in $n-1$ generators $X_1, \dots, X_{n-1}$. Here
$X_i\in H^1(W)$ is the cohomology class represented by the map
$f_i: W\to S^1$ given by $f_i(u_1, \dots, u_n) = u_iu_{n}^{-1}\in
S^1$, where $i=1, \dots, n-1$.

Theorem \ref{thm2} would follow once we have shown that:
\begin{enumerate}
\item[(i)] The image of (\ref{induced}) equals the balanced
subalgebra $B_\ell^\ast$;

\item[(ii)] The kernel of (\ref{induced}) coincides with the ideal
$I$ (described above) after the identification $H^\ast(W)=E(X_1,
\dots, X_{n-1})$.
\end{enumerate}

To prove statement (i) we note that the involution $\tau:
M_\ell\to M_\ell$ is the restriction of a natural involution
$\tau: W\to W$ given by formula (\ref{invol}).

\begin{Lemma}\label{lm1}
The action of the involution $\tau^\ast$ on the integral
cohomology groups $H^i(W)$ and $H^i(W, M_\ell)$ coincides with
multiplication by $(-1)^i$, for any $i=0, 1, \dots, n-1$.
\end{Lemma}

We postpone the proof of Lemma \ref{lm1} and continue with the proof
of Theorem~\ref{thm2}. Consider the homomorphism
\begin{eqnarray}\label{phii}
\psi_i: H^i(W, M_\ell) \to H^i(W)
\end{eqnarray}
induced by the inclusion $W\to (W, M_\ell)$. The exact sequence of
the pair $(W, M_\ell)$ gives a short exact sequence
\begin{eqnarray}\label{exseq}
0\to \coker(\psi_i) \to H^i(M_\ell)\stackrel\delta\to
\ker(\psi_{i+1})\to 0.
\end{eqnarray}
The involution $\tau^\ast$ acts on the exact sequence (\ref{exseq})
and according to Lemma \ref{lm1} the action of $\tau^\ast$ is
$(-1)^i$ on $\coker(\psi_i)$ and $(-1)^{i+1}$ on $\ker
\psi_{i+1}$. This implies that the balanced subalgebra $B^i_\ell$
coincides with the image of $\coker(\psi_i)$ in $H^i(M_\ell)$. In
other words, we have established statement (i).

To prove statement (ii) we show that the image of $\psi_i$ (see
(\ref{phii})) is generated by the monomials $X_{r_1}\dots X_{r_i}$
where $1\leq r_1<r_2<\dots<r_i<n$ are such that the subset $\{r_1,
\dots, r_i\}\cup \{n\}$ is long.

For a subset $J=\{r_1, r_2, \dots, r_i, n\}$ where $1\leq r_1<r_2
<\dots <r_i<n$, we denote the corresponding product $X_{r_1}\dots
X_{r_i}$ by $X_J$. The cohomology class $X_J\in H^i(W)$ is
realized by the continuous map $\psi_J: W\to T^{i}$ given by
$$(u_1, \dots, u_n) \mapsto (\frac{u_{r_1}}{u_n},
\frac{u_{r_2}}{u_n}, \dots \frac{u_{r_i}}{u_n}) \in T^i.$$ The
preimage of the point $(1, 1, \dots, 1)\in T^i$ under $\psi_J$
equals the variety of all states of the robot arm such that all
links $r_1, \dots, r_i$ are parallel to the last link $n$. Denote
this variety by $W_J\subset W$. Clearly, $W_J$ is diffeomorphic to
a torus of dimension $n-|J|=n-1-i$. The submanifold $W_J$ is
Poincar\'e dual to the cohomology class $X_J$. For obvious
geometric reasons, a subset $J\subset \{1, \dots, n\}$ is long if
and only iff the corresponding submanifold $W_J\subset W$ is
disjoint from $M_\ell$.

In the following commutative diagram
\begin{eqnarray}\xymatrix @C=3.5pc {H^i(W,M_\ell)\ar[r]^{\psi_i}\ar[d]_\simeq& H^i(W)\ar[d]^\simeq \\
H_{n-1-i}(W-M_\ell)\ar[r]^-{\phi_{n-1-i}}&H_{n-1-i}(W)}\label{cd}\end{eqnarray}
the vertical maps are Poincar\'e duality isomorphisms. Here
$\phi_{n-1-i}$ is induced by the inclusion $W-M_\ell\to W$. In \S
5 of \cite{FS} it was shown that the image of
$\phi_{n-1-i}$ is generated by the homology classes $[W_J]\in
H_{n-1-i}(W)$ of the submanifolds $W_J$ where $J\subset \{1,
\dots, n\}$ runs over all long subsets of cardinality $i+1$ with
$n\in J$. Since $\{n\}$ is not long, the case $i=0$ is excluded by
our assumption. Hence, the image of $\psi_{i}$ is generated by
cohomology classes $X_J$ where $J\subset \{1, \dots, n\}$ is long,
$|J|=i+1$ and $n\in J$. This completes the proof.
\end{proof}

\begin{proof}[Proof of Lemma \ref{lm1}] Our statement concerning
the action of the involution $\tau$ on $H^i(W)$ is quite obvious
since $W$ is a torus $T^{n-1}=S^1\times S^1\times\dots\times S^1$
and the involution acts as the complex conjugation on each of the
circles.

To prove our claim concerning the action of the involution on $H^i(W,
M_\ell)$ consider
the robot arm distance map (the negative square of the length of the dotted line shown on Figure \ref{arm})
\begin{eqnarray}\label{fel}
f_\ell: W\to \R, \quad f_\ell(u_1, \dots, u_n) = -
\left|\sum_{i=1}^n l_i u_i\right|^2.
\end{eqnarray}
The maximum of $f_\ell$ is achieved on the moduli space
$M_\ell=f_\ell^{-1}(0) \subset W$. The critical points of $f_\ell$
lying in $W-M_\ell$ are exactly collinear configurations which can
be labeled by long subsets $J\subset \{1, \dots, n\}$, see Lemma
7 of \cite{FS}. The involution $\tau: W\to W$ preserves the values
of $f_\ell$ and the fixed points of $\tau$ are collinear
configurations of the robot arm (see \S 4, 5 of \cite{FS}).
Choose a real number $a$ such
$$-\left(   \sum_{i\in J} l_i - \sum_{i\notin J}l_i   \right)^2 <
a< 0$$ for any long subset $J\subset \{1, \dots, n\}$. Then the
manifold (with boundary) $N=f_\ell^{-1}([a,0])$ contains $M_\ell$ as a deformation
retract, see text after Corollary 10 in \cite{FS}. The set $W^a=f_\ell^{-1}((-\infty, a])$
is a compact manifold with boundary and there is a natural isomorphism  $H^i(W,M_\ell)\to H^i(W^a, \partial W^a)$
commuting with the action of the involution. Lemma \ref{lm1} follows once we show that $\tau^\ast x=(-1)^i x$ for $x\in H^i(W^a, \partial W^a)$.

Let $u\in H_{n-1}(W^a, \partial W^a)$ denote the fundamental class. $\tau$ acts as multiplication
by $-1$ on the tangent space of every fixed point (see (14) in \cite{FS}) which means that $\tau$ has degree $(-1)^{n-1}$.
Hence,  $\tau_\ast(u) = (-1)^{n-1}u$.
For a cohomology class $x\in H^i(W^a, \partial W^a)$ we have
$$(-1)^{n-1} x\cap u = x\cap \tau_\ast(u) = \tau_\ast(\tau^\ast x\cap u) = (-1)^{n-1-i}\tau^\ast x \cap u\in H_{n-1-i}(W^a).$$
Here the last equality uses Theorem 4 from \cite{FS}. This shows (by Poincar\'e duality) that $\tau^\ast x = (-1)^i x$ completing the proof.
\end{proof}

\section{Poincar\'e duality defect}\label{secpd}

If the length vector $\ell=(l_1, \dots, l_n)$ is not generic then the space $M_\ell$ has finitely many singular points.
In this section we show
that for a non generic $\ell$ the space $M_\ell$ fails to satisfy the Poincar\'e duality and we describe explicitly the defect.
Denote by \begin{eqnarray}
K^i_\ell\subset H^i(M_\ell)\end{eqnarray}
the set of all cohomology classes $u\in H^i(M_\ell)$ such that $$uw=0 \quad \mbox{for any}\quad w\in H^{n-3-i}(M_\ell).$$
It is obvious that $K^\ast_\ell=\oplus K^i_\ell$ is an ideal in $H^\ast(M_\ell)$.

Recall, that by
\cite{FS}, for $n>3$ one has $$H^{0}(M_\ell)=H^{n-3}(M_\ell)=\Z$$ if and only if the sets $\{n-2, n-1\}$ and $\{n\}$  are short with respect to $\ell$.

\begin{Theorem}\label{defect} Suppose that $\ell=(l_1, l_2, \dots, l_n)$
is such that $0<l_1\le l_2\le \dots \le l_n$ and
$H^{0}(M_\ell)=H^{n-3}(M_\ell) = \Z$. Then one has $$K_\ell^\ast\subset B^\ast_\ell,$$ i.e. all cohomology classes in $K_\ell^\ast$ are balanced.
Moreover, $K^i_\ell$ viewed as a free abelian group, has a free basis given by the monomials of the form
\begin{eqnarray}\label{prod}
X_{r_1}X_{r_2}\dots X_{r_i}\end{eqnarray}
where $1\le r_1< r_2 < \dots < r_i < n$ are such that the subset \begin{eqnarray}\label{med1}
\{r_1, r_2, \dots, r_i, n\}\subset \{1, \dots, n\}\end{eqnarray} is median with respect to $\ell$.
\end{Theorem}

\begin{proof}[Proof of Theorem \ref{defect}]
Replace the length vector $\ell$ by a generic vector $\ell'$ which is obtained by enlarging slightly the longest link $l_n$, i.e. $\ell'=(l'_1, l'_2, \dots, l'_n)$ where
$l'_i=l_i$ for $i=1, \dots, n-1$ and $l'_n=l_n+\epsilon$. We assume that $\epsilon >0$ is so small that the vectors $\ell$ and $\ell'$ have the same set of short and long subsets  $J\subset \{1, \dots, n\}$ while any subset  $J\subset \{1, \dots, n\}$ containing $n$ which is median with respect to $\ell$ becomes long with respect
to $\ell'$.

Below we construct a ring epimorphism
\begin{eqnarray}
F: H^\ast(M_\ell) \to H^\ast(M_{\ell'})
\end{eqnarray}
having the following properties:

(a) $F$ is an isomorphism in dimensions $0$ and $n-3$;

(b) one has
\begin{eqnarray}
F(X_i)=X'_i, \quad i=1, \dots, n-1
\end{eqnarray}
where $X_i\in H^1(M_\ell)$,  $X'_i\in H^1(M_{\ell'})$ denote generators (\ref{generators});

(c) the rank of the kernel of $F: H^i(M_\ell) \to H^i(M_{\ell'})$ equals the number of subsets $J\subset \{1, \dots, n\}$ of cardinality $i+1$ containing $n$ which are median with respect to $\ell$.

Let us show that the existence of such $F$ implies the statement of Theorem \ref{defect}.
We claim that $K^\ast_\ell$ coincides with the kernel of $F$. Indeed,
if $u\in H^i(M_\ell)$ does not lie in $K^i_\ell$ then $uw\not=0$ for some $w\in H^{n-3-i}(M_\ell)$ and applying $F$ we find that
$F(u)F(w)\not=0$, i.e. $u\notin \ker F$. Conversely, if $F(u)\not=0$ then $F(u)w'\not=0$ for some $w'\in H^{n-3-i}(M_{\ell'})$ (since $M_{\ell'}$ is a closed manifold without singularities). Using the surjectivity of $F$ we can write $w'=F(w)$ for some $w\in H^i(M_\ell)$ and therefore $F(uw)\not=0$ implying $uw\not=0$, i.e. $u\notin K^i_\ell$.

Using property (b) and the description of the balanced subalgebra given in Theorem \ref{thm2} we see that all monomials (\ref{prod}), such that the set (\ref{med1}) is median with respect to $\ell$, are mapped trivially by $F$. Property (c) implies that such monomials form the whole kernel $\ker F$, proving Theorem \ref{defect}.

To construct $F$ with properties mentioned above consider the robot arm with links $l_1, l_2, \dots, l_{n-1}$ and its configuration space $T^{n-1}=S^1\times S^1 \times \dots \times S^1$, ($n-1$ times). Consider the map
\begin{eqnarray}
f: T^{n-1} \to \R^2\quad \mbox{ given by} \quad  f(u_1, \dots, u_{n-1}) = \sum_{i=1}^{n-1}l_iu_i
\end{eqnarray}
where $u_i\in S^1$. Denote by $U$ the preimage of the positive real axis $\R_+\subset \R^2$ with respect to $f$, i.e. $U=f^{-1}(\R_+)$.
It is a manifold of dimension $n-2$ which is diffeomorphic to the moduli space of all non-closed shapes of the robot arm with $n-1$ links $l_1, \dots, l_{n-1}$. In particular we see that $U$ is canonically embedded into the torus $T^{n-1}$ and therefore we have coordinate projections $\phi_i: U\to S^1$ where $i=1, \dots, n-1$. The cohomology class $X_i\in H^1(U)$ induced by $\phi_i$ \lq\lq measures\rq\rq\, the angle between the link number $i$ and the positive real axis.

The function $g=f|U:\, U\to \R_+$ is Morse (see \cite{Ha} and \cite{FS}) and its critical points are configurations $c_I\in U$ with $u_i=\pm 1$ which can be labeled by subsets
$I\subset \{1, \dots, n-1\}$ satisfying $\sum_{i\in I}l_i > \sum_{i\notin I} l_i$. The critical value $g(c_I)$ equals
$$g(c_I) = \sum_{i\in I}l_i - \sum_{i\notin I} l_i$$ and the Morse index of $g$ at $c_I$ equals $|I|-1$, see \cite{Ha}, Theorem 3.2.

We see that $l_n\in \R_+$ is a critical value of $g$ and the critical level set $g^{-1}(l_n)\subset U$ is exactly the polygon space $M_\ell$.
The critical points of $g$ lying in $g^{-1}(l_n)$ can now be relabeled by subsets $J\subset \{1, \dots, n\}$ containing $n$ which are median with respect to $\ell$. The Morse index of such a critical point is $n-|J|-1$.

Similarly, one has $g^{-1}(l'_n) = M_{\ell'}$. It is a regular level of $g$. Note that the interval $(l_n, l'_n)$ contains no critical values of $g$ as follows from our assumption concerning $\epsilon= l'_n-l_n$.

Consider the preimage $V=g^{-1}[l_n, l'_n]$.
Clearly $V$ is a cobordism $\partial V = M_\ell \sqcup M_{\ell'}$
which has Morse type singularities along the lower boundary $M_{\ell}$.
A standard construction of Morse theory (using the negative gradient flow) gives a deformation retraction of
$V$ onto $M_\ell$; in other words, the inclusion $M_\ell\to V$ is a homotopy equivalence.
We denote by $F$ the map defined by the commutative diagram
\[\xymatrix{ & H^i(V) \ar[ld]_\simeq \ar[rd]& \\
H^i(M_\ell)\ar[rr]^F & & H^i(M_{\ell'})}\]
where both inclined arrows are induced by the inclusions $M_\ell \to V$ and $M_{\ell'}\to V$.

Clearly $F$ is a ring homomorphism having property (b) (see above), i.e. $F(X_i)=X'_i$ for all $i=1, \dots, n-1$.

The homological exact sequence of the pair $(V, M_{\ell'})$, in which one replaces $H^\ast(V)$ by $H^\ast(M_\ell)$, gives the following exact sequence
\begin{eqnarray}
\dots \to H^i(V,M_{\ell'})\stackrel \alpha\to H^i(M_\ell) \stackrel F\to H^i(M_{\ell'}) \stackrel \delta\to H^{i+1}(V, M_{\ell'})\to \dots
\end{eqnarray}
Here $H^i(V, M_{\ell'})$ is a free abelian group of rank equal the number of critical points of index $\dim V - i=n-2-i$ lying in $g^{-1}(l_n)=M_\ell$. In other words, the rank of $H^i(V, M_{\ell'})$ equals the number of subsets $J\subset \{1, \dots, n\}$ with $n\in J$, $|J|=i+1$, which are median with respect to $\ell$.
Hence we see that the rank of the kernel of $F: H^i(M_\ell) \to H^i(M_{\ell'})$ is less or equal than the number of such median subsets. On the other hand, using $F(X_i)=X'_i$ and Theorem \ref{thm2} applied to $M_\ell$ and $M_{\ell'}$, we find that all monomials (\ref{prod}) with (\ref{med1}) lie in this kernel.
Here we use the fact that any subset $J$ containing $n$, which is median with respect to $\ell$, is long with respect to $\ell'$.

We conclude that $\alpha$ is injective and $\delta=0$. In other words, we obtain that $F$ is an epimorphism and its kernel has rank as stated in (c).

The remaining property (a) follows from (c) using our assumption that $H^0(M_\ell)=H^{n-3}(M_\ell)=\Z$.
\end{proof}

\section{Proof of Theorem \ref{thmain}}\label{prfthm1}

Given $\ell\in A^{(n-1)}$, denote by $\cals^0 (\ell)$ the family of subsets $J\subset \{1, \dots, n\}$ such that $n\notin J$ and $J$ is short with respect to $\ell$. Similarly, denote by $\cals^1 (\ell)$ the family of subsets $J\subset \{1, \dots, n\}$ such that $n\in J$ and $J$ is short with respect to $\ell$.

\begin{Lemma}\label{gen} Let $\ell, \ell'\in A^{(n-1)}$ be two ordered length vectors, i.e. $\ell=(l_1, \dots, l_n)$ where $0<l_1\le \dots\leq l_n$ and similarly for $\ell'$.
If for some permutation $\sigma: \{1, \dots, n\}\to \{1, \dots, n\}$ with $\sigma(n)=n$ and for $\nu=0$ or $1$ one has $\sigma(\cals^\nu(\ell)) = \cals^\nu(\ell')$ then $\cals^\nu(\ell)=\cals^\nu( \ell')$.
 \end{Lemma}
 \begin{proof} Let $i$ be the smallest number such that $\sigma(i) >i$. Then $\sigma(k)=k$ for all $k<i$.
 Let $j=\sigma(i)$ (note that $j>i$) and $\sigma'=\alpha\circ \sigma$ where $\alpha=(ij)$ denotes the transposition of $i$ and $j$.
Observe that $\sigma(k)=\sigma'(k)$ for all $k\not=i,\,  r$ where $r>i$ is such that $\sigma(r)=i.$
We want to show that
\begin{eqnarray}\label{given}
\sigma(\cals^\nu(\ell))=\cals^\nu(\ell'),\end{eqnarray}
where $\nu=0$ or $\nu=1$, implies that \begin{eqnarray}\label{impl}\sigma'(\cals^\nu(\ell))=\cals^\nu(\ell').\end{eqnarray}
Once this implication has been established, Lemma \ref{gen} follows by induction since after several iterations the permutation $\sigma$ will be replaced by the identity permutation.

If $J\subset \cals^\nu(\ell)$ is such that either $i\in J$ and $r\in J$, or $i\notin J$ and $r\notin J$ then $\sigma(J)=\sigma'(J)$ and hence $\sigma'(J) \in \cals^\nu(\ell')$ by (\ref{given}).  Consider now the case when  $i\in J$ and $r\notin J$.
Then $\sigma'(J)$ is obtained from $\sigma(J)$ by adding $i$ and removing $j$. Since $\sigma(J)$ is short with respect to $\ell'$ and $l'_j\geq l'_i$, we see that $\sigma'(J)$ is short with respect to $\ell'$.
The remaining possibility is $i\notin J$ and $r\in J$. Then $\sigma'(J) = \sigma(I)$ where $I$ is obtained from $J$ by adding $i$ and removing $r$.
If $J$ is short with respect to $\ell$ then $I$ is short as well (since $r>i$) and hence the set $\sigma(I)=\sigma'(J)$ is short with respect to $\ell'$ .

This argument shows that $\sigma'(\cals^\nu(\ell))\subset \cals^\nu(\ell').$ Since by (\ref{given}) the sets $\cals^\nu(\ell)$ and $\cals^\nu(\ell')$ have equal cardinality, this implies (\ref{impl}).
 \end{proof}

Next we state a theorem of J. Gubeladze \cite{Gu} playing a key
role in the proof of our main results. We are thankful to Lucas
Sabalka who brought this theorem to our attention.

Let $R$ be a commutative ring. Consider the ring
 $R[X_1, \dots, X_{m}]$
of polynomials in variables $X_1, \dots, X_m$ with coefficients in
$R$. A {\it {monomial ideal}} $I\subset R[X_1, \dots, X_m]$ is an
ideal generated by a set of monomials $X_1^{a_1} \dots X_m^{a_m}$
where $a_i\in \Z$, $a_i\geq 0$. The factor-ring $R[X_1, \dots,
X_m]/I$ is called {\it a discrete Hodge algebra}, see \cite{DCEP}.

One may view the variables $X_1, \dots, X_m$ as elements of the
discrete Hodge algebra $R[X_1, \dots, X_m]/I$. The main question
is whether it is possible to recover the relations $X_1^{a_1}\dots
X_m^{a_m}=0$ in $R[X_1, \dots, X_m]/I$ using only intrinsic
algebraic properties of the Hodge algebra. This question is known
as {\it the isomorphism problem} for commutative monoidal rings,
it was solved in \cite{Gu}:

\begin{Theorem}[J. Gubeladze \cite{Gu}, Theorem 3.1]\label{gu}
Let $R$ be a commutative ring and $\{X_1, \dots, X_m\}$, $\{Y_1,
\dots, Y_{m'}\}$ be two collections of variables. Assume that
$I\subset R[X_1, \dots, X_m]$ and $I'\subset R[Y_1, \dots,
Y_{m'}]$ are two monomial ideals such that $I\cap \{X_1, \dots,
X_m\}=\emptyset$ and $I'\cap \{Y_1, \dots, Y_{m'}\}=\emptyset$ and
factor-rings
$$
R[X_1, \dots, X_m]/I \simeq R[Y_1, \dots, Y_{m'}]/I'
$$
are isomorphic as $R$-algebras. Then $m=m'$ and there exists a
bijective mapping
$$\Theta: \{X_1, \dots, X_m\} \to \{Y_1, \dots, Y_m\}$$
transforming $I$ into $I'$. \qed
\end{Theorem}

\begin{proof}[Proof of Theorem \ref{thmain}] Suppose that $\ell=(l_1,\dots, l_n)\in A$ and $\ell'=(l'_1,
\dots, l'_n) \in A$ are two length vectors such that
there exists a graded ring isomorphism between the integral
cohomology algebras $H^\ast(M_\ell)\to H^\ast(M_{\ell'})$
commuting with the action of the involution $\tau$, see
(\ref{invol}), (\ref{involstar}).
We assume that $\ell=(l_1, \dots,
l_n)$ and $\ell'=(l'_1, \dots, l'_n)$ are ordered, i.e. $l_1\le \dots \le l_n$,
and $l'_1\le \dots \le l'_n$.
Our goal is to show (as Theorem \ref{thmain} states) that then the
length vectors $\ell, \ell'$ lie in the same stratum.

Theorem \ref{thmain} is obviously true in the trivial case $n=3$. The case $n=4$ is also trivial: in this case one has a complete classification of
spaces $M_\ell$ (which are one-dimensional) and their topology is fully determined by Betti numbers $b_0(M_\ell)$ and $b_1(M_\ell)$.
Hence we will assume below that $n>4$.

First we show that it is enough to prove Theorem \ref{thmain} assuming that
\begin{eqnarray}\label{maincase}
H^0(M_\ell)=H^{n-3}(M_\ell) =\Z.
\end{eqnarray}
To this end we recall the main result of \cite{FS} describing explicitly the Betti numbers $b_k(M_\ell)$. Let $a_k(\ell)$ and $\tilde a_k(\ell)$
denote correspondingly the number of short and median subsets
$J\subset \{1, \dots, n\}$ satisfying $n\in J$ and $|J|=k+1$. Then
\begin{eqnarray}\label{betti}
b_k(M_\ell) = a_k(\ell)+ a_{n-3-k}(\ell) + \tilde a_k(\ell),
\end{eqnarray}
see \cite{FS}, Theorem 1.
Using (\ref{betti}) we obtain that (\ref{maincase}) is equivalent to two equations
\begin{eqnarray}\label{1}
a_0(\ell) +a_{n-3}(\ell)+ \tilde a_0(\ell) =1, \quad a_0(\ell)+a_{n-3}(\ell) +\tilde a_{n-3}(\ell)  =1.
\end{eqnarray}

If the subsets $\{n\}$ and $\{n-2, n-1\}$ are short with respect to $\ell$
then clearly $a_0(\ell)=1$ while $\tilde a_0(\ell)=a_{n-3}(\ell)=\tilde a_{n-3}(\ell) =0$; hence (\ref{maincase}) holds in this case.

Consider now all possibilities when the assumptions discussed in the previous paragraph are violated, i.e. when either $n$ is not short or the set
$\{n-2, n-1\}$ is not short.
The table below displays Betti numbers $b_0=b_0(M_\ell)$,  $b_1=b_1(M_\ell)$ and $b_{n-3}=b_{n-3}(M_\ell)$ and the set of all short subsets ${\mathcal S}=\{J\}$ where $J\subset
\{1, \dots, n\}$ in each case. It follows that Betti numbers $b_0(M_\ell)$, $b_1(M_\ell)$ and $b_{n-3}(M_\ell)$ fully characterize each case since they determine the entire set of short subsets ${\mathcal S}$, compare Lemma \ref{stra}.

This shows that Theorem \ref{thmain} holds if  (\ref{maincase}) fails to hold since then the stratum containing $\ell$ (equivalently, the family of short subsets with respect to $\ell$, by Lemma \ref{stra}) can be determined knowing only the Betti numbers in dimensions $0$, $1$ and $n-3$.

\vspace{0.1in}
\begin{center}
\begin{tabular}{|c|c|c|c|c|}
  \hline
Case & $b_0$ & $b_1$ & $b_{n-3}$ & Short subsets $\mathcal S$ \\ \hline\hline
 $n$ long & $0$ & $0$ & $0$ & $J\subset \{1, \dots, n-1\}$\\ \hline
 $n$ median & $1$ & $0$ & $0$ &  $J \varsubsetneq \{1, \dots, n-1\}$ \\ &&&&  \\ \hline
 $n$ short  & $2$ &$2n-6$ &$ 2$& $J\subset \{1, \dots, n\}$, \\ $\{n-2, n-1\}$ long &&&&$J$ contains at most \\ &&&& one of $n-2, n-1, n$ \\ \hline
  $n$ short  & $1$ & $2n-6$ & $2$ & $J\varsubsetneq \{1, \dots, n-3, n\}$,\\ $\{n-2, n-1\}$ median &&&&  \\ $\{n-2, n\}$ long &&&& $J\subset\{1, 2, \dots, n-3, n-2\}$, \\
    $\{n-1, n\}$ long &&&& $J\subset\{1, 2, \dots, n-3, n-1\}$\\
  \hline
 $n$ short  & $1$ & $2n -5$ & $2$ & $J\varsubsetneq \{1, \dots, n-3, n\}$,\\ $\{n-2, n-1\}$ median &&&&  \\ $\{n-2, n\}$ median &&&& $J\varsubsetneq\{1, 2, \dots, n-3, n-1\}$, \\
    $\{n-1, n\}$ long &&&& $J\subset\{1, 2, \dots, n-3, n-2\}$ \\
  \hline
$n$ short  & $1$ & $2n-4$ & $2$ & $J\varsubsetneq \{1, \dots, n-3, n\}$,\\ $\{n-2, n-1\}$ median &&&&  \\ $\{n-2, n\}$ median &&&& $J\varsubsetneq\{1, 2, \dots, n-3, n-1\}$, \\
    $\{n-1, n\}$ median &&&& $J\varsubsetneq\{1, 2, \dots, n-3, n-2\}$ \\  \hline

\end{tabular}
\end{center}
\vspace{0.1in}

Below we give a proof of Theorem \ref{thmain} assuming that (\ref{maincase}) holds for $\ell$ and $\ell'$. This is equivalent to the requirement that the sets
$\{n\}$ and $\{n-2, n-1\}$ are short with respect to $\ell$ and $\ell'$.

Since the isomorphism $f: H^\ast(M_\ell)\to H^\ast(M_{\ell'})$
commutes with the action of the involution $\tau^\ast$ it gives a graded
algebra isomorphism $B^\ast_\ell \to B^\ast_{\ell'}$ between the
balanced subalgebras. In section \S \ref{secpd} we introduced subgroups $K_\ell^\ast\subset H^\ast(M_\ell)$ and $K_{\ell'}^\ast\subset H^\ast(M_{\ell'})$
measuring the failure of the Poincar\'e duality. Clearly $f$ must map $K_\ell^\ast$ onto $K_{\ell'}^\ast$ since these groups are defined using the multiplicative structure of the cohomology algebras.

The structure of $B_\ell^\ast$ and $K_\ell^\ast$ was described in Theorems \ref{thm2} and \ref{defect}. In particular we know that $K^\ast_\ell\subset B^\ast_\ell$. We obtain the following diagram of isomorphisms
$$
\begin{array}{ccc}
H^\ast(M_\ell) & \stackrel f\to & H^\ast(M_{\ell'})\\
\uparrow  &&  \uparrow\\
B^\ast_\ell & \stackrel f\to & B^\ast_{\ell'}\\
\uparrow  &&  \uparrow\\
K^\ast_\ell & \stackrel f\to & K^\ast_{\ell'}
\end{array}
$$
where the vertical maps are inclusions.

Now we apply Theorem \ref{thm2} to the tensor products $\Z_2\otimes
B^\ast_\ell$ and $\Z_2\otimes B^\ast_{\ell'}$. By Theorem  \ref{thm2} they are
 discrete Hodge
algebras $$\Z_2\otimes
B^\ast_\ell= \Z_2[X_1, \dots, X_{n-1}]/L, \quad  \Z_2\otimes
B^\ast_{\ell'} = \Z_2[X'_1, \dots,
X'_{n-1}]/L'.$$ Here $L$ is the monomial ideal generated by the
squares $X_r^2$ (for each $r=1,\dots, n-1$) and by the monomials
$X_{r_1}X_{r_2}\dots X_{r_p}$ for each sequence $1\leq r_1<\dots
<r_p<n$ such that the subset $\{r_1, \dots, r_p\}\cup
\{n\}\subset \{1, \dots, n\}$ is long with respect to $\ell$. The
ideal $L'$ is defined similarly with $\ell'$ replacing $\ell$.

Clearly, $X_i\in L$ if and only if $\{i, n\}$ is long. In this
case one has $X_j\in L$ for all $i\le j<n$. Let $i=i(\ell)$ denote the
smallest index with the property $X_i\in L$. We obtain that the
balanced subalgebra $\Z_2\otimes B^\ast_\ell\subset
H^\ast(M_\ell)$ is isomorphic to $\Z_2[X_1, \dots, X_{i-1}]/I$
where $I$ is the monomial ideal generated by the squares $X_r^2$
(for each $r=1,\dots, i-1$) and by the monomials
$X_{r_1}X_{r_2}\dots X_{r_p}$ for each sequence $1\leq r_1<\dots
<r_p<i$ such that the subset $\{r_1, \dots, r_p\}\cup \{n\}\subset
\{1, \dots, n\}$ is long with respect to $\ell$.

Similarly, one defines the number $i'=i(\ell')$ and the balanced
subalgebra $\Z_2\otimes B^\ast_{\ell'}$ is isomorphic to a
discrete Hodge algebra $\Z_2[X'_1, \dots, X'_{i'-1}]/I'$ where
$L'$ is defined similarly to $I$ with $\ell'$ replacing $\ell$.

Note that ${\rk}\,  B^1_\ell = i -1$ and $\rk\,  B^1_{\ell'} = i'
-1$ and therefore one has $i=i'$.

Now we are in the situation of Theorem \ref{gu} of J. Gubeladze.
We have two discrete Hodge algebras $\Z_2[X_1, \dots, X_{i-1}]/I$
and $\Z_2[X'_1, \dots, X'_{i-1}]/I'$ which are isomorphic. Besides
we know that $X_j\notin I$, $X'_j\notin I'$ for $j=1, \dots, i-1$.
By Theorem \ref{gu} there exists a bijection $\Theta: \{X_1,
\dots, X_{i-1}\} \to \{X'_1, \dots, X'_{i-1}\}$ transforming $I$
into $I'$. This means that there exists a bijection $\theta: \{1,
\dots, i-1\} \to \{1, \dots, i-1\}$ such that for any sequence
$1\leq r_1<\dots <r_p<i$ the subset $\{r_1, \dots, r_p, n\}$ is
short or median with respect to $\ell$ if and only if $\{\theta(r_1), \dots,
\theta(r_p), n\}$ is short or median with respect to $\ell'$. Note that a
short or median subset (with respect to either length vector $\ell$ or
$\ell'$) which contains $n$ cannot contain any element $j$
satisfying $i\leq j<n$. Let $\sigma: \{1, \dots, n\} \to \{1,
\dots, n\}$ be defined by
$$\sigma(j) = \left\{
\begin{array}{ll}
\theta(j) & \mbox{for} \, 1\leq j <i,\\
j & \mbox{for} \, i\le j\leq n.
\end{array}
\right.
$$
We see that the length
vectors $\ell$ and $\sigma(\ell')=(l'_{\sigma(1)}, \dots,
l'_{\sigma(n)})$ have the same family of long subsets containing $n$, or, equivalently, the same family of short subsets not containing $n$.

As in Lemma \ref{gen}, denote by $\cals^0 (\ell)$ the family of subsets $J\subset \{1, \dots, n\}$ such that $n\notin J$ and $J$ is short with respect to $\ell$. Denote by $\cals^1 (\ell)$ the family of subsets $J\subset \{1, \dots, n\}$ such that $n\in J$ and $J$ is short with respect to $\ell$.

We see that
$\cals^0(\ell) = \cals^0(\sigma(\ell'))= \sigma^{-1}(\cals^0(\ell'))$. By Lemma \ref{gen} we obtain $\cals^0(\ell) = \cals^0(\ell')$.

Theorem \ref{defect} combined with Theorem \ref{thm2} imply that the algebra $\Z_2\otimes (B^\ast_\ell/K^\ast_\ell)$ is a discrete Hodge algebra
$$\Z_2\otimes (B^\ast_\ell/K^\ast_\ell) = \Z_2[X_1, \dots, X_{n-1}]/\tilde L$$
where $\tilde L\supset L$ is the monomial ideal generated by the
squares $X_r^2$ (for each $r=1,\dots, n-1$) and by the monomials
$X_{r_1}X_{r_2}\dots X_{r_p}$ for each sequence $1\leq r_1<\dots
<r_p<n$ such that the subset $\{r_1, \dots, r_p\}\cup
\{n\}\subset \{1, \dots, n\}$ is long or median with respect to $\ell$.

Similarly, $\Z_2\otimes (B^\ast_{\ell'}/K^\ast_{\ell'}) = \Z_2[X_1, \dots, X_{n-1}]/\tilde L'.$
The ideal
$\tilde L'$ is defined analogously to $\tilde L$ with vector $\ell'$ replacing $\ell$.

Repeating the arguments above, applying Theorem \ref{gu} of Gubeladze to the isomorphism $\Z_2\otimes (B^\ast_{\ell}/K^\ast_{\ell})\to \Z_2\otimes (B^\ast_{\ell'}/K^\ast_{\ell'})$ and using again Lemma \ref{gen} we find that length vectors $\ell$ and $\ell'$ have identical sets of short subsets containing $n$, i.e. $\cals^1(\ell)=\cals^1(\ell')$.

Lemma \ref{lm4} below implies that $\ell$ and $\ell'$ lie in the same stratum of $A$. This completes the proof.
\end{proof}

\begin{Lemma}\label{lm4} Two length vectors $\ell, \ell'\in A^{(n-1)}$ lie in the same stratum of $A$ if and only if
the following condition holds: a subset $J\subset \{1, \dots, n\}$ containing $n$ is
short (or median) with respect to $\ell$ if and only if it is short (or median, correspondingly) with
respect to $\ell'$.
\end{Lemma}
\begin{proof}
In view of Lemma \ref{stra}, we have to show that the condition mentioned in the statement of the lemma implies that $\ell$ and $\ell'$ have identical families of short subsets. This condition specifies short, median and long subsets containing $n$. However, if $n\notin J$ then $J$ is short, median or long iff $\bar J$ is long, median or short, correspondingly.
\end{proof}

\section{Proofs of Theorems \ref{thmain21} and \ref{thmain31}}\label{prfthm34}

 Let $\ell$ be a generic ordered length vector with $n>4$. The planar polygon space $\bar M_\ell$ is defined as the factor of $M_\ell$ with respect to involution (\ref{invol}). Consider the factor-map $\pi: M_\ell\to \bar M_\ell$. Since we assume that $\ell$ is generic, $\pi$ is a twofold covering map. Let $$w_1(\ell)\in H^1(\bar M_\ell;\Z_2)$$ denote its first Stiefel - Whitney class.

  \begin{Lemma}\label{lm7}
 There is a graded algebra isomorphism
 \begin{eqnarray}\label{factor}
 \Z_2\otimes B_\ell^\ast\to H^{\ast}(\bar M_\ell;\Z_2)/(w_1(\ell)).
 \end{eqnarray} Here $B^\ast_\ell\subset H^\ast(M_\ell)$ denotes the balanced subalgebra, see \S \ref{balanced}.
\end{Lemma}
\begin{proof} Consider the induced homomorphism
\begin{eqnarray}\label{piast}
\pi^\ast: H^\ast(\bar M_\ell;\Z_2) \to H^\ast(M_\ell;\Z_2).
\end{eqnarray}
We claim that:

{\it (a) the kernel of $\pi^\ast$ coincides with the ideal $(w_1(\ell))$ generated by the class $w_1(\ell)$ and

(b) the image of $\pi^\ast$ coincides with
the image of the balanced subalgebra $B^\ast_\ell\subset H^\ast(M_\ell;\Z)$ under the coefficient homomorphism $\Z\to \Z_2$. }

A presentation of the cohomology ring $H^*(\bar
M_\ell;\Z_2)$ in terms of generators and relations is given in \cite[Corollary 9.2]{HK}:
$$
H^*(\bar M_\ell;\Z_2)= \Z_2[R,V_1,\dots ,V_{n-1}] /
{\cal I}_{\ell} $$
  where $R$ and $V_1, \dots, V_{n-1}$ are of degree 1
  and ${\cal I}_{\ell}$ is the ideal generated by the three families of elements:
  \bigskip

  \noindent
  (R1) \, $V_i^2+ RV_i$ \, \, where  $i=1,\dots ,n-1$,  \smallskip

   \noindent (R2)  \, ${\displaystyle \prod^{\ }_{i\in S} V_i}$  \,
  where $S\subset \{1,\dots ,n-1\}$ is such that $S\cup \{n\}$ is long,

    \noindent(R3)  \,  ${\displaystyle \sum^{\ }_{S\subset L} R^{|L-S|-1}\prod_{i\in
  S}V_i }$ \, \, where \,   $L\subset \{1,\dots ,n-1\}$ is $L$ long.

  \noindent The symbol $S$  in (R3) runs over all subsets of $L$ including the empty set. By  (R2) a term of the sum in (R3) vanishes if $S\cup \{n\}$ is long.

The class $R$ coincides with $w_1(\ell)\in H^1(M_\ell;\Z_2)$,
the first Stiefel-Whitney class of the regular $2$-cover
$\pi:M_\ell\to\bar M_\ell$ \cite[Prop.~9.3]{HK}. Therefore,
$\pi^*(R)=0$. The classes $V_i$, constructed in
\cite[Section~5]{HK}, are Poincar\'e dual to the sub-manifold of
$\bar M_\ell$ of those configurations where the $i$-th link is
parallel to the last one. Therefore, $\pi^*(V_i)=X_i$ where $X_i$ are generators (\ref{generators}). Hence the image of $\pi^\ast$ lies in the balanced
subalgebra. Note that every relation (R3) is a multiple of $R$ since the term $$R^{|L-S|-1}\prod_{i\in  S}V_i $$ either vanishes (if $S\cup\{n\}$ is long, see (R2)) or the exponent $|L-S|-1$ is positive. Therefore adding a relation $R=0$ to (R1) - (R3) transforms it into the presentation of the balanced subalgebra, see
Theorem \ref{thm2}.  This completes the proof of our statements (a) and (b) concerning the homomorphism (\ref{piast}) which imply Lemma \ref{lm7}.
\end{proof}

\begin{proof}[Proof of Theorem \ref{thmain21}] Assume that $\ell$ is a generic ordered length vector.
First we note that the Stiefel - Whitney class $w_1(\ell) \in H^1(\bar M_\ell;\Z_2)$ can be uniquely determined using intrinsic algebraic properties of the algebra
$H^\ast_\ell=H^\ast(\bar M_\ell;\Z_2)$. Consider the squaring map
\begin{eqnarray}
H^1_\ell \to H^2_\ell, \quad \mbox{where}\quad v\mapsto v^2, \quad v\in H^1_\ell.
\end{eqnarray}
It is a group homomorphism and we claim that for $n>4$ the Stiefel - Whitney class $u=w_1(\ell)\in H^1_\ell$ is the unique cohomology
class $u\in H^1_\ell$ satisfying
\begin{eqnarray}\label{rel}
v^2 = vu, \quad \mbox{for all}\quad v\in H^1_\ell.
\end{eqnarray}
We will use the presentation of the ring $H^*(\bar
M_\ell;\Z_2)$ in terms of generators and relations described in the proof of Lemma~\ref{lm7}.
The class $R$ represents $w_1(\ell)\in H^1(\bar M_\ell;\Z_2)$,
see \cite[Prop.~9.3]{HK}.
Now, relation (R1) (see proof of Lemma \ref{lm7}) tells us that (\ref{rel}) holds for all generators $V_i$. It is obviously true for $R=w_1(\ell)$ as well. Hence
(\ref{rel}) holds for $u=R=w_1(\ell)$ and for all $v$.
To show uniqueness of $u$ satisfying (\ref{rel}) (here we use our assumption $n>4$)
suppose that $u'\in H^1_\ell$ is another class satisfying (\ref{rel}). Then $w=u-u'\in H^1_\ell$ has the property that its product with any class $v\in H^1_\ell$
vanishes. But we know that classes of degree one generate $H^\ast_\ell$ as an algebra (it follows from  \cite[Corollary 9.2]{HK})
and therefore, using Poincar\'e duality, we obtain that $w=0$ assuming that the class $w$ is not of top degree, i.e. if $1< n-3$.

Suppose now that for two generic length vectors $\ell, \ell'$ there is a graded algebra isomorphism $f: H^\ast(\bar M_\ell;\Z_2)\to H^\ast(\bar M_{\ell'};\Z_2)$, $n>4$.  From the result of the previous paragraph it follows that $f(w_1(\ell))=w_1(\ell')$. We obtain that the factor rings $H^{\ast}(\bar M_\ell;\Z_2)/(w_1(\ell))$ corresponding to $\ell$ and $\ell'$ are isomorphic. By Lemma \ref{lm7} we obtain that the rings $\Z_2\otimes B^\ast_\ell$ and $\Z_2\otimes B^\ast_{\ell'}$ are isomorphic. Now, the arguments of the proof of Theorem \ref{thmain} show that $\ell$ and $\ell'$ lie in the same chamber of $A$.
\end{proof}

\begin{proof}[Proof of Theorem \ref{thmain31}]
By \cite[Theorem 6.4]{HK}, the cohomology ring of the spatial polygon space
$H^*(N_\ell;\Z_2)$ admits the same presentation as the presentation of
$H^*(\bar
M_\ell;\Z_2)$ described above with the only difference that now
the variables $V_i$ and $R$  are of degree $2$. Hence, the existence of an algebraic isomorphism $H^\ast(N_\ell;\Z_2) \to H^\ast(N_{\ell'};\Z_2)$ implies the existence of an algebraic isomorphism $H^\ast(\bar M_\ell;\Z_2) \to H^\ast(\bar M_{\ell'};\Z_2)$ and our statement now follows from Theorem \ref{thmain21}.

If the integral cohomology rings $H^\ast(N_\ell)$ and $H^\ast(N_{\ell'})$ are isomorphic then the cohomology rings with $\Z_2$ coefficients are isomorphic as well, since the integral cohomology has no torsion and therefore the $\Z_2$-cohomology is the tensor product of the integral cohomology with $\Z_2$.
\end{proof}

\section{Proof of Theorem \ref{thmain2}}\label{prfthm2}

\begin{Lemma}\label{normal} A length vector $\ell$ is normal if and only
if $H^1(M_\ell) = B^1_\ell$, i.e. when all cohomology classes of degree
one are balanced. Hence for a normal length vector the balanced
subalgebra $B^\ast_\ell$ coincides with the subalgebra of
$H^\ast(M_\ell)$ generated (as a unital algebra) by the set $H^1(M_\ell)$ of one-dimensional
cohomology classes.
\end{Lemma}

\begin{proof}[Proof of Lemma \ref{normal}] Without loss of generality we may assume that $\{n\}$ is short since otherwise $\ell$ is normal and $H^1(M_\ell)=0$.
We will use arguments of the proof of Theorem
\ref{thm2}. Our aim is to show that $\ell$ is normal if and only
if the kernel of the homomorphism $\psi_2$ (see (\ref{phii}) with
$i=2$) vanishes. Then the desired equality $H^1(M_\ell) =
B^1_\ell$ would follow from exact sequence (\ref{exseq}) combined
with the arguments of the proof of Theorem \ref{thm2}. By looking at the
commutative diagram (\ref{cd}) we find that $\psi_2$ is injective
if and only if $\phi_{n-3}$ is injective. By Corollary 10 of \cite{FS}
the kernel of $\psi_2$ vanishes if and only if there are no long
subsets $J\subset \{1, \dots, n-1\}$ with $|J|=3$. But this is
 equivalent to normality of $\ell$ as we have explained right after Definition \ref{defnormal}.
 \end{proof}

\begin{Lemma}\label{normal2} Assume that $\ell$ is a length vector satisfying $b_0(M_\ell)=1=b_{n-3}(M_\ell)$.
Then $\ell$ is normal if and
only if the $(n-3)$-fold cup product
\begin{eqnarray}\label{n-3fold}
\mathop{\underbrace{H\otimes H\otimes \dots \otimes H}}_{n-3\, \,
\mbox{times}}\to H^{n-3}(M_\ell), \quad \mbox{where}\quad
H=H^1(M_\ell)
\end{eqnarray}
vanishes.
\end{Lemma}
\begin{proof}
If $\ell$ is normal then $H=B^1_\ell$ by Lemma \ref{normal}. We will use the description of the balanced subalgebra given by Theorem \ref{thm2}.
We obtain that $H$ is generated by the classes $X_1, X_2, \dots, X_{n-1}$ and a product $X_{r_1}X_{r_2}\dots X_{r_{n-3}}$ vanishes if and only if the set $\{r_1, \dots, r_{n-3}, n\}$ is long. Since $\ell$ is normal, the set $\{n-3, n-2, n-1\}$ is median or long which implies that any subset $J\subset \{1, \dots, n-1\}$ of cardinality $2$ is short. This proves that the product (\ref{n-3fold}) vanishes.

Conversely, suppose that $\ell$ is not normal. Then by Lemma
\ref{normal} the group $H=H^1(M_\ell)$ contains a non-balanced
element $u$, i.e. $\tau^\ast(u) \not=-u$. Hence
$u_0=u+\tau^\ast(u)$ is nonzero and satisfies
$\tau^\ast(u_0)=u_0$.
We claim that there exists $v\in
H^{n-4}(M_\ell)$ such that $vu_0=w\in H^{n-3}(M_\ell)$ is nonzero.
If $\ell$ is generic this claim follows from Poincar\'e duality.
If $\ell$ is not generic, we proceed as in the proof of Theorem \ref{defect}
replacing $\ell$ by a close-by generic length-vector $\ell'$ which is obtained by enlarging slightly the longest link $l_n$, i.e. $\ell'=(l'_1, l'_2, \dots, l'_n)$ where
$l'_i=l_i$ for $i=1, \dots, n-1$ and $l'_n=l_n+\epsilon$. Here $\epsilon >0$ is so small that the vectors $\ell$ and $\ell'$ have the same set of short and long subsets  $J\subset \{1, \dots, n\}$ while any subset  $J\subset \{1, \dots, n\}$ containing $n$ which is median with respect to $\ell$ becomes long with respect
to $\ell'$. In the proof of Theorem \ref{defect} we constructed
a ring epimorphism
$
F: H^\ast(M_\ell) \to H^\ast(M_{\ell'})
$
which is an isomorphism in the top dimension $n-3$ and is such that the kernel of $F$ coincides with the subgroup $K^\ast_\ell$ which lies in $B^\ast_\ell$.
Since $u_0\notin B^1_\ell$ we obtain $F(u_0)\not=0$ and applying the Poincar\'e duality (note that $M_{\ell'}$ is a closed orientable manifold without singularities) we find that $F(u_0)v_0\not=0$ for some $v_0\in H^{n-4}(M_{\ell'})$. Since $F$ is an epimorphism we may write $v_0=F(v)$ and we see that
$u_0v\not=0$ since $F(u_0v)\not=0$.

From our assumptions $n>3$ and $b_0(M_\ell)=1=b_{n-3}(M_\ell)$ it follows that the sets $\{n-2, n-1\}$ and $\{n\}$ are short and then Theorem
\ref{thm2} implies that there are no balanced classes in the top dimensional cohomology group $H^{n-3}(M_\ell)\simeq \Z$. Since $\tau(w)=\pm w$ we find
$\tau^\ast(w)=-(-1)^{n-3}w.$ Applying
$\tau^\ast$ we obtain $\tau^\ast(v)u_0=(-1)^{n-4}w$ and therefore
\begin{eqnarray}\label{hence}
v_0u_0=2w\quad \mbox{where}\quad v_0 = v+(-1)^{n-4}\tau^\ast(v).
\end{eqnarray}
Since $v_0$ is balanced (as is obvious from its definition (\ref{hence})) it can
be expressed as a polynomial of degree $n-4$ in $H$ (by Theorem
\ref{thm2}).
Hence we see that the $(n-3)$-fold cup-product
(\ref{n-3fold}) is nonzero.
\end{proof}

\begin{proof}[Proof of Theorem \ref{thmain2}] We shall assume that $n>3$ as for $n=3$ our statement is trivial.
Suppose that $\ell$ is
normal and there exists a graded algebra isomorphism
$f: H^\ast(M_\ell)\to H^\ast(M_{\ell'})$. We want to show that
$\ell'$ is also normal. Assume first that $b_0(M_\ell)=1=b_1(M_\ell)$. Then we may apply Lemma \ref{normal2} which characterizes normality in intrinsic
algebraic terms of the cohomology algebra. Hence $\ell'$ must be also normal.

If the condition $b_0(M_\ell)=1=b_1(M_\ell)$ is violated for $\ell$ it is also violated for $\ell'$.
Examining the table in the proof of Theorem \ref{thmain} we conclude that the Betti numbers $b_0(M_\ell)$ and $b_1(M_\ell)$ determine then the family of short subsets and hence the stratum of $\ell$ (by Lemma \ref{stra}).

Next we apply Lemma \ref{normal} which implies that the isomorphism $f: H^\ast(M_\ell)\to H^\ast(M_{\ell'})$
maps the balanced subalgebra $B^\ast_\ell$ isomorphically onto $B^\ast_{\ell'}$ since in the case of a normal length vector the balanced subalgebra coincides with the subalgebra generated by one-dimensional cohomology classes. One repeats the arguments of the proof of Theorem \ref{thmain} to conclude that then
$\ell$ and $\ell'$ lie in the same stratum.
\end{proof}


\end{document}